\documentclass[12pt]{amsart}
\usepackage{graphics}
\usepackage{amsfonts,amsmath}
\begin{document}

 \newtheorem{thm}{Theorem}[section]
 \newtheorem{cor}[thm]{Corollary}
 \newtheorem{lem}[thm]{Lemma}{\rm}
 \newtheorem{prop}[thm]{Proposition}

 \newtheorem{defn}[thm]{Definition}{\rm}
 \newtheorem{assumption}[thm]{Assumption}
 \newtheorem{rem}[thm]{Remark}
 \newtheorem{ex}{Example}
\numberwithin{equation}{section}
\def\la{\langle}
\def\ra{\rangle}
\def\glexe{\leq_{gl}\,}
\def\glex{<_{gl}\,}
\def\e{{\rm e}}

\def\x{\mathbf{x}}
\def\P{\mathbf{P}}
\def\S{\mathbf{S}}
\def\h{\mathbf{h}}
\def\by{\mathbf{y}}
\def\bz{\mathbf{z}}
\def\F{\mathcal{F}}
\def\R{\mathbb{R}}
\def\T{\mathbf{T}}
\def\N{\mathbb{N}}
\def\D{\mathbf{D}}
\def\V{\mathbf{V}}
\def\U{\mathbf{U}}
\def\K{\mathbf{K}}
\def\Q{\mathbf{Q}}
\def\M{\mathbf{M}}
\def\oM{\overline{\mathbf{M}}}
\def\O{\mathbf{O}}
\def\C{\mathbb{C}}
\def\P{\mathbb{P}}
\def\Z{\mathbb{Z}}
\def\H{\mathcal{H}}
\def\A{\mathbf{A}}
\def\V{\mathbf{V}}
\def\AA{\overline{\mathbf{A}}}
\def\B{\mathbf{B}}
\def\c{\mathbf{c}}
\def\L{\mathcal{L}}
\def\bS{\mathbf{S}}
\def\H{\mathcal{H}}
\def\I{\mathbf{I}}
\def\Y{\mathbf{Y}}
\def\X{\mathbf{X}}
\def\G{\mathbf{G}}
\def\f{\mathbf{f}}
\def\z{\mathbf{z}}
\def\v{\mathbf{v}}
\def\y{\mathbf{y}}
\def\d{\hat{d}}
\def\bx{\mathbf{x}}
\def\bI{\mathbf{I}}
\def\y{\mathbf{y}}
\def\g{\mathbf{g}}
\def\w{\mathbf{w}}
\def\b{\mathbf{b}}
\def\a{\mathbf{a}}
\def\v{\mathbf{v}}
\def\u{\mathbf{u}}
\def\q{\mathbf{q}}
\def\e{\mathbf{e}}
\def\s{\mathcal{S}}
\def\cc{\mathcal{C}}
\def\co{{\rm co}\,}
\def\tg{\tilde{g}}
\def\tx{\tilde{\x}}
\def\tg{\tilde{g}}
\def\tA{\tilde{\A}}

\def\supmu{{\rm supp}\,\mu}
\def\supp{{\rm supp}\,}
\def\cd{\mathcal{C}_d}
\def\cok{\mathcal{C}_{\K}}
\def\cop{COP}
\def\vol{{\rm vol}\,}

\title{volume of slices and sections of the simplex in closed form}
\author{Jean B. Lasserre}
\thanks{This work was supported by a grant of the PGMO program of the 
{\it Fondation Math\'ematique Jacques Hadamard} (FMJH), Paris)}

\address{LAAS-CNRS and Institute of Mathematics\\
University of Toulouse\\
LAAS, 7 avenue du Colonel Roche\\
31077 Toulouse C\'edex 4, France\\
Tel: +33561336415}
\email{lasserre@laas.fr}

\date{}
\begin{abstract}
Given a vector $\a\in\R^n$, we provide an alternative and direct proof for the formula
of the volume of sections
$\Delta\cap \{\x: \a^T\x <= t\}$ 
and slices  $\Delta\cap\{\x:\a^T\x = t\}$, $t\in\R$,  of the simplex $\Delta$.
For slices the formula has already been derived but as a by-product of the construction of univariate B-Splines.
One goal of the paper is to also show how simple and powerful can be the Laplace transform technique
to derive closed form expression for some multivariate integrals.
It also complements some previous results obtained for the hypercube $[0,1]^n$.
\end{abstract}

\maketitle

\section{Introduction}

In Marichal and Mossinghof \cite{marichal} the authors have provided a closed-form expression of slices and slabs of the unit 
hypercube cube $[0,1]^n$.
In the interesting discussion on the history and applications (e.g. in probability and statistical mechanics) of this problem, they mention how similar but earlier results had been already proved, notably by P\'olya in his PhD dissertation. In \cite{marichal}
the authors' proof relies on a {\it signed} simplicial decomposition
of the unit cube and the inclusion-exclusion principle whereas P\'olya's  approach was different and related the volume  to some
{\it sinc} integrals as also did Borwein et al. \cite{borwein} much later. For more details the interested reader is referred to
\cite{marichal} and the references therein. 

For the simplex  one can find several contributions in the literature 
for integrating polynomials and  defining cubatures formula;
see for instance the recent work of Baldoni et al. \cite{baldoni} and the many references therein.
But concerning the slice of a simplex it turns out that a formula for the volume of the slice has already been derived ... as a by-product 
in the construction of univariate B-splines! Indeed as explained in
Micchelli \cite[pp. 150--153]{micchelli}, in their construction of univariate B-splines of degree $n-1$,
Curry and Schoenberg showed that they are interpreted as volumes of slices of a $n$-simplex! 
In the description\footnote{{\tt http://epubs.siam.org/doi/abs/10.1137/1.9781611970067.ch4}} of Chapter 4
 in \cite{micchelli} one may even read
{\it ``This chapter explores the powerful idea of generating multivariate smooth piecewise polynomials as 
the volume of ÒslicesÓ of polyhedra."}

\subsection*{Contribution}
The goal of this note is to provide a relatively simple and direct proof  for volumes of sections and slices of the simplex
without taking a detour via the theory of univariate B-splines. It also shows how powerful and ``easy" 
can be Laplace techniques for such a purpose.
We consider sections and slices of the canonical simplex $\Delta:=\{\x:\e^T\x \leq 1\}$ (where $\e\in\R^n$ is the vector of ones). That is, given a vector $\a$ of the unit sphere $\S^{n-1}$ and some $t\in\R$, we want to compute the $n$-dimensional (resp. $(n-1)$-dimensional) Lebesgue volume of the sets
\begin{eqnarray}
\label{a1}
\Theta(\a,t)&:=&\Delta\,\cap\,\{\,\x\in\R^n: \a^T\x\leq t\,\}\\
\label{a2}
\S(\a,t)&:=&\Delta\,\cap\,\{\,\x\in\R^n: \a^T\x\,=\,t\,\}
\end{eqnarray}
Notice that in contrast to the unit hypercube $[0,1]^n$, the simplex $\Delta$ is {\it not} a cartesian product
and the ``canonical" signed simplicial decomposition of the sliced hypercube used in \cite{marichal}
has no analogue for the simplex $\Delta$. So for the simplex a different approach is needed and
ours is based on the Laplace transform technique already used in 
our previous work in \cite{las-zeron} for computing a certain class of multivariate integrals.

In particular and in contrast to the case of the sliced hypercube, some special care is needed 
when some weights $a_i$, $i\in I$ (for some subset $I\subset\{1,\ldots,n\}$), are identical,
in which case the generic formula for the volume of the section of the simplex degenerates.

Finally, the result of \cite{marichal} for the unit hypercube can be retrieved from
our results either (a) directly by applying the same Laplace technique to the context of the hypercube or (b)
in two steps by first decomposing the unit hypercube into the union of $2^n$ simplices and then applying our 
``sliced-simplex formula" to each of the simplices.

\section{Main result}

Given a scalar $x\in\R$,
the notation $(x)_+$ stands for $\max[0,x]$. Let us first recall the following result already proved in
\cite{barvinok} and \cite{las-zeron} for the canonical simplex $\Delta$.
\begin{lem}[\cite{barvinok},\cite{las-zeron}]
\label{lem1}
Let $\c:=(c_1,\ldots,c_n)\in\R^n$, $c_0:=0$, with $c_i\neq c_j$ for every distinct pair $(i,j)$. Then
\begin{equation}
\label{lem1-1}
\int_\Delta \exp(-\c^T\x)\,d\x\,=\,\sum_{i=0}^n \frac{\exp(-c_i)}{\prod_{j\neq i}(c_i-c_j)}.
\end{equation}
\end{lem}
We will use Lemma \ref{lem1} to prove the following result for the sliced simplex and 
a section of the simplex as well.
\begin{thm}
\label{th1}
Let $\a=(a_1,\ldots,a_n)\in\S^{n-1}$, $a_0:=0$, and assume that
$a_i\neq a_j$ for any pair $(i,j)$ with $i\neq j$. Then
\begin{eqnarray}
\label{th1-1}
{\rm vol}\,(\Theta(\a,t))&=&\frac{1}{n{\rm !}}\left(\sum_{i=0}^n\frac{(t-a_i)_+^n}{\prod_{j\neq i}(a_j-a_i)}\,\right).\\
\label{th1-2}
{\rm vol}\,(\S(\a,t))&=&\frac{1}{(n-1){\rm !}}\left(\sum_{i=0}^n\frac{(t-a_i)_+^{n-1}}{\prod_{j\neq i}(a_j-a_i)}\,\right).
\end{eqnarray}
\end{thm}
\begin{proof}
We first assume that $\a\geq0$.
Let $f:\R_+\to\R$ be the Lebesgue volume of $\Theta(\a,t)$ so that 
$f(t)=0$ for every $t<0$. Let $F:\C\to\C$ its Laplace transform $\mathcal{L}[f]$
\[\lambda\mapsto F(\lambda)\,:=\,\int_0^\infty \exp(-\lambda t)\,f(t)\,dt,\]
whose domain is $\mathcal{D}:=\{\lambda\in\C:\Re(\lambda)>0\}$.
Hence with $\lambda$ real, $\lambda>0$,
\begin{eqnarray*}
F(\lambda)&=&\int_0^\infty \exp(-\lambda t)\,f(t)\,dt\\
&=&\int_0^\infty\left(\int_{\Delta\cap\{\x:\a^T\x\leq t\}}d\x\right) \exp(-\lambda t)\,dt\\
&=&\int_{\Delta}\left(\int_{\a^T\x}^\infty\exp(-\lambda t)\,dt\right)\,d\x\quad\mbox{[by Fubini-Tonelli]}\\
&=&\frac{1}{\lambda}\,\int_{\Delta}\exp(-\lambda \,\a^T\x)\,d\x\,=\,\frac{1}{\lambda^{n+1}}\left(\sum_{i=0}^n \frac{\exp(-\lambda \,a_i)}
{\prod_{j\neq i}(a_j-a_i)}\right)\\
&=&\sum_{i=0}^n \frac{1}{n{\rm !}\,\prod_{j\neq i}(a_j-a_i)}\cdot\underbrace{\frac{\Gamma(n+1)}{\lambda^{n+1}}}_{\mathcal{L}[t^n]}\cdot
\exp(-\lambda\,a_i)
\end{eqnarray*}
where we have used (\ref{lem1-1}). 
Notice that the function 
\[\lambda\mapsto h(\lambda):=\sum_{i=0}^n \frac{1}{n{\rm !}\,\prod_{j\neq i}(a_j-a_i)}\cdot\frac{\Gamma(n+1)}{\lambda^{n+1}}\cdot
\exp(-\lambda\,a_i)\]
is analytic on $\mathcal{D}$ and coincides with $F(\lambda)$ on $(0,+\infty)\subset\mathcal{D}$.
Therefore by the Identity Theorem for analytic functions (see e.g. Freitag and Busam \cite{freitag}) $F(\lambda)=h(\lambda)$ on $\mathcal{D}$; 

Next recall that $\lambda\mapsto \exp(-\lambda c)\mathcal{L}[t^n](\lambda)$ is the Laplace transform of $u_c(t)(t-c)^n$ with $u_c(t)$ being the Heavyside function 
$t\mapsto u_c(t)=1$ if $t\geq c$ and $u_c(t)=0$ otherwise; see e.g. \cite{widder}. Therefore,
\[f(t)\,=\,\frac{1}{n{\rm !}}\left(\sum_{\{\,i:a_i\leq t\,\}}\frac{(t-a_i)^n}{\prod_{j\neq i}(a_j-a_i)}\,\right)\,=\,
\frac{1}{n{\rm !}}\left(\sum_{i=0}^n\frac{(t-a_i)_+^n}{\prod_{j\neq i}(a_j-a_i)}\,\right),\]
which is the desired result (\ref{th1-1}). 

Next, notice that if $n$ is even then $f$ is differentiable and if $n$ is odd $f$ is differentiable for all $t\not\in\cup_i\{a_i\}$, with
\[f'(t)\,=\,\frac{1}{(n-1){\rm !}}\left(\sum_{i=0}^n\frac{(t-a_i)_+^{n-1}}{\prod_{j\neq i}(a_j-a_i)}\,\right),\quad t\not\in\{a_0,a_1,\ldots,a_n\}.\]
Moreover, for all $t\not\in\{a_0,a_1,\ldots,a_n\}$, the volume of the section
$\{\x:\a^T\x=t\}\cap\Delta$ is given by:
\begin{eqnarray*}
{\rm vol}(\Delta\cap\{\x:\a^T\x=t\})
&=&\Vert\a\Vert\,f'(t),\quad t\not\in\{a_0,a_1,\ldots,a_n\}\\
&=&\frac{1}{(n-1){\rm !}}\left(\sum_{i=0}^n\frac{(t-a_i)_+^{n-1}}{\prod_{j\neq i}(a_j-a_i)}\,\right),
\end{eqnarray*}
since $\a\in\S^{n-1}$ and where the first equality follows from a formula provided in \cite{lasserre} for the $(n-1)$-volume of the $(n-1)$-dimensional 
facet of an arbitrary polytope. Finally, by using a simple continuity argument,
formula (\ref{th1-1}) remains  valid at those points $t=a_i$, $i=0,\ldots,n$.\\

Let us now consider the case where $a_i<0$, $i\in I$, for some nonempty subset $I\subset\{1,\ldots,n\}$. 
Hence we may and will suppose that $a_1<a_2<\ldots <a_n$ with $a_1<0$.
Let
\[t^*\,=\,\min\,\{\,\a^T\x:\x\in\Delta\,\}\,=\,\min\,\{a_i: i\in I\,\}\,=\,a_1\,<\,0,\]
and observe that $f(t)=0$ whenever $t\leq t^*$. Therefore let
$t\mapsto g(t):=f(t+t^*)$ so that $g(t)=0$ if $t\leq 0$. Its Laplace transform reads
\begin{eqnarray*}
\int_0^\infty\exp(-\lambda t)\,g(t)\,dt&=&\int_0^\infty\exp(-\lambda t)f(t+t^*)\,dt\\
&=&\int_\Delta\left(\int_{\a^T\x-t^*}^\infty\exp(-\lambda t)\,dt\right)\,d\x\\
&=&\exp(\lambda t^*)\,F(\lambda)\\
&=&\sum_{i=0}^n \frac{\exp(-\lambda\,(a_i-t^*))}{n{\rm !}\,\prod_{j\neq i}(a_j-a_i)}\,\underbrace{\frac{\Gamma(n+1)}{\lambda^{n+1}}}_{\mathcal{L}[t^n]},
\end{eqnarray*}
provided that $a_i\neq a_j$ for all distinct pairs $(i,j)$; and so
\[g(t)\,=\,\frac{1}{n{\rm !}}\left(\sum_{\{\,i:a_i\leq t+t^*\,\}}\frac{(t-a_i+t^*)^n}{\prod_{j\neq i}(a_j-a_i)}\,\right)\,=\,
\frac{1}{n{\rm !}}\left(\sum_{i=0}^n\frac{(t-(a_i-a_1))_+^n}{\prod_{j\neq i}(a_j-a_i)}\,\right),\]
provided that $a_i\neq a_j$ for all distinct pairs $(i,j)$. But then as $f(t)=g(t-a_1)$ we obtain the desired result (\ref{th1-1}).
Finally (\ref{th1-2}) is obtained as before.
\end{proof}
So if the $a_i$'s are now ordered with increasing values
$a_1 < a_2 <\cdots <0< \cdots a_n$ then from (\ref{th1-1}) one can see that the volume function $t\mapsto \Theta(\a,t)$ of the sliced simplex is a piecewise polynomial of degree $n$ 
on $[a_1,a_2]\cup [a_2,a_3]\cup\cdots\cup [a_{n-1},a_n]$ and constant equal to zero on
$(-\infty,a_1]$ and equal to $1/n{\rm !}$ on $[a_n,+\infty)$.

As explained in Micchelli's book \cite[p. 50--54]{micchelli}, formula (\ref{th1-2}) for the volume of slices of $\Delta$ 
was already known to Curry and Schoenberg in their construction of (piecewise polynomials) 
$B$-Splines. Starting from  the formula of the B-Splines of degree $n-1$ 
their proved that it can be interpreted as volumes of slices of $\Delta$. 

\subsection{The case of identical weights}
Let us indicate what happens to formula (\ref{th1-1}) if e.g. $a_k=a_\ell$ for some pair $(k,\ell)$.
An obvious way to see what happens is to perform a perturbation analysis.
So suppose that $a_i\neq a_j$ for all distinct pairs $(i,j)\neq (k,\ell)$ 
and let $a_\ell:=a_k+\epsilon$ with $\epsilon>0$ sufficiently small so
that now $a_i\neq a_j$ for all pairs $(i,j)$ (including the pair $(k,\ell)$). Then formula
(\ref{th1-1}) reads
\[
\frac{1}{n{\rm !}}\left(\sum_{i\neq k;i\neq \ell}\frac{(t-a_i)_+^n}{\prod_{j\neq i}(a_i-a_j)}\,\right)\]
\[+\frac{1}{\epsilon\,n{\rm !}}\left(\frac{(t-a_k)_+^n}{\prod_{j\neq k;j\neq \ell}(a_j-a_k)}
-\frac{(t-a_k-\epsilon)_+^n}{\prod_{j\neq k;j\neq \ell}(a_j-a_k-\epsilon)}
\right)\]
and therefore when $\epsilon\to 0$ one obtains
\begin{eqnarray*}
{\rm vol}\,(\Theta(\a,t))&=&
\frac{1}{n{\rm !}}\left(\sum_{i\neq k;i\neq \ell}\frac{(t-a_i)_+^n}{\prod_{j\neq i}(a_i-a_j)}\,\right)\\
&&-\frac{1}{(n-1){\rm !}}
\left(\frac{(t-a_k)_+^{n-1}}{\prod_{j\neq k;j\neq \ell}(a_j-a_k)}\right)
\end{eqnarray*}
for all $t\neq a_k$.

A similar analysis can be done if several coefficients $a_{i_1},a_{i_2},a_{i_k}$ are identical to some
value $s$, in which case a term of the form $\frac{1}{(n-k){\rm !}}
\frac{(t-s)_+^{n-k}}{\prod_{j\neq i_1,\ldots,i_k}(a_j-s)}$ appears.

\subsection{For an arbitrary simplex}
For an arbitrary simplex  $\boldsymbol{\Omega}$ with set of vertices $V:=\{\v_1,\ldots,\v_{n+1}\}$ 
and  $\a\in\S^{n-1}$ such that $\a^T\v\neq \a^T\w$ for any pair of distinct vertices $(\v,\w)\in V^2$,
formula (\ref{th1-1}) becomes
\[f(t)\,=\,\frac{1}{n{\rm !}}\left(\sum_{\v\in V}\frac{(t-\a^T\v)_+^n}{\displaystyle\prod_{\w\in V;\w\neq\v}\a^T(\v-\w)}\,\right).\]
Similarly, at every point $t\not\in\{\a^T\v: \v\in V\}$, the volume of the section
$\{\x:\a^T\x=t\}\cap\boldsymbol{\Omega}$ is given by:
\[{\rm vol}(\boldsymbol{\Omega}\cap\{\x:\a^T\x=t\})\,=\,\frac{1}{(n-1){\rm !}}\left(\sum_{\v\in V}\frac{(t-\a^T\v)_+^{n-1}}{\displaystyle\prod_{\w\in V;\w\neq\v}\a^T(\v-\w)}\,\right).\]

\subsection{The unit hypercube $[0,1]^n$}

We now use the same Laplace technique to retrieve a formula for the unit hypercube $[0,1]^n$ already provided in
Marichal and Mossinghof \cite{marichal}. So define the sets
\begin{eqnarray*}
\B&:=&\{\,\x\in\R^n_+\::\: \x\,\leq\,\boldsymbol{1}\,\}\\
\B_\a(t)&:=&\{\,\x\in\R^n_+\::\: \x\,\leq\,\boldsymbol{1};\: \a^T\x \leq t\},\quad t\in\R.
\end{eqnarray*}
First assume that $0<\a\in\S^{n-1}$ and let $f:\R_+\to\R$ be the Lebesgue volume of $\B_\a(t)$ and $F:\C\to\C$ its Laplace transform $\mathcal{L}[f]$
\[\lambda\mapsto F(\lambda)\,:=\,\int_0^\infty \exp(-\lambda t)\,f(t)\,dt,\]
where here its domain is $\mathcal{D}:=\{\lambda\in\C:\Re(\lambda)>0\}$.
With $\lambda$ real, $\lambda>0$,
\begin{eqnarray*}
F(\lambda)&=&\int_0^\infty \exp(-\lambda t)\,f(t)\,dt\\
&=&\int_{\B}\left(\int_{\a^T\x}^\infty\exp(-\lambda t)\,dt\right)\,d\x,\quad\mbox{[by Fubini-Tonelli]}\\
&=&\frac{1}{\lambda}\,\int_{\B}\exp(-\lambda \,\a^T\x)\,d\x
\,=\,\frac{\prod_{i=1}^n (1-\exp(-\lambda \,a_i))}{\lambda^{n+1}\prod_{i=1}^n a_i}\\
&=&\frac{1}{\lambda^{n+1}\prod_{i=1}^n a_i}\sum_{U\subseteq\{1,\ldots,n\}} (-1)^{\# U}
\exp(-\lambda\, \a^T\boldsymbol{1}_U)\\
&=&\frac{1}{n{\rm !}\,\prod_{i=1}^n a_i}\sum_{U\subseteq\{1,\ldots,n\}} (-1)^{\# U}
\frac{n{\rm !}}{\lambda^{n+1}}\,\exp(-\lambda\, \a^T\boldsymbol{1}_U)
\end{eqnarray*}
where $\boldsymbol{1}_U\in\{0,1\}^n$ is the indicator of the set $U$. Therefore proceeding as in the proof
of Theorem \ref{th1},
\[f(t)\,=\,\frac{1}{n{\rm !}\,\prod_{i=1}^n a_i}\sum_{U\subseteq\{1,\ldots,n\}} (-1)^{\# U}(t-\a^T\boldsymbol{1}_U)_+^n\]
which is formula (3) in \cite[Theorem 1]{marichal}. Similarly, the volume for a slice of the hypercube can be deduced 
by taking the derivative of $f(t)$.
Finally the case where $\a\in\S^{n-1}$ has some negative entries can be recovered as we did for the simplex.


\begin{thebibliography}{las}
\bibitem{barvinok}
A. Barvinok, Computing the volume, counting integral points, and exponential sums,
Discrete Comput. Geom. {\bf 10}, pp. 123--141, 1993.
\bibitem{baldoni}
V. Baldoni, N. Berline, J. De Loera, M. K\"oppe, M. Vergne. How to integrate a polynomial over a simplex,
Math. Comp. {\bf 80}, pp. 297--325, 2011.
\bibitem{borwein}
D. Borwein, J.M. Borwein, B.A. Mares Jr. Multi-variable sinc integrals and volumes of polyhedra,
Ramanujan J. {\bf 6}, pp. 189--208, 2002. 
\bibitem{freitag}
E. Freitag and R. Busam. {\em Complex Analysis}, Second Edition, Springer-Verlag, Berlin, 2009.
\bibitem{lasserre}
J.B. Lasserre. An analytical expression and an algorithm for the volume of a convex polytope in $\R^n$,
J. Optim. Theory Appl. {\bf 39}, pp. 363--377, 1983
\bibitem{las-zeron}
J.B. Lasserre, E.S. Zeron. Solving a class of multivariate integration problems via Laplace techniques,
Applicationes Math. {\bf 28}, pp. 391--405, 2001.
\bibitem{marichal}
J.-L. Marichal, M.J. Mossinghoff. Slices, slabs, and sections of the unit hypercube,
{\it Online J. Analytic Combinatorics} {\bf 3}, \#1, 2008
\bibitem{micchelli}
C.A. Micchelli. {\it Mathematical Aspects of Geometric Modeling}, 
CBMS-NSF Regional Conference Series, SIAM, Philadelphia, 1995.
\bibitem{polya}
G. P\'olya. On a few questions in probability theory ad some definite integrals related to them,
PhD Thesis, E\"otv\"os Lor\'and University, 1912.
\bibitem{widder}
D.V. Widder. {\em The Laplace Transform}, Princeton University Press, Princeton, 1946.
\end{thebibliography}
\end{document}